\documentclass[12pt,reqno]{amsart}
\usepackage{amssymb}
\usepackage{graphicx}
\usepackage{xcolor}
\usepackage{longtable}
\usepackage{float}

\usepackage[all]{xy}
\DeclareFontFamily{U}{mathb}{\hyphenchar\font45}
\DeclareFontShape{U}{mathb}{m}{n}{
      <5> <6> <7> <8> <9> <10> gen * mathb
      <10.95> mathb10 <12> <14.4> <17.28> <20.74> <24.88> mathb12
      }{}
\DeclareSymbolFont{mathb}{U}{mathb}{m}{n}
\DeclareMathSymbol{\righttoleftarrow}{3}{mathb}{"FD}

\theoremstyle{plain}
\newtheorem{prop}{Proposition}[section]
\newtheorem{theo}[prop]{Theorem}

\newtheorem{lemm}[prop]{Lemma}
\theoremstyle{remark}

\theoremstyle{definition}

\newtheorem{rema}[prop]{Remark}

\newtheorem{exam}[prop]{Example}
\numberwithin{equation}{section}

\def\rP{{\mathrm P}}

\def\rB{{\mathrm B}}

\def\fD{{\mathfrak D}}

\def\fK{{\mathfrak K}}

\def\fQ{{\mathfrak Q}}

\def\fS{{\mathfrak S}}

\def\fS{{\mathfrak S}}

\def\bA{{\mathbb A}}
\def\bG{{\mathbb G}}
\def\bP{{\mathbb P}}
\def\bQ{{\mathbb Q}}

\def\bZ{{\mathbb Z}}
\def\bR{{\mathbb R}}

\def\bN{{\mathbb N}}

\def\rH{{\mathrm H}}

\DeclareMathOperator{\rk}{rk}

\def\lra{\longrightarrow}

\def\res{\mathrm{res}}

\def\Pic{\mathrm{Pic}}
\def\End{\mathrm{End}}

\def\Aut{\mathrm{Aut}}

\def\GL{\mathsf{GL}}

\def\PGL{\mathsf{PGL}}

\def\Ext{\mathrm{Ext}}

\def\lim{\mathrm{lim}}

\def\Ker{\mathrm{Ker}}

\def\rN{\mathrm{N}}
\def\rM{\mathrm{M}}

\def\PL{\mathrm{PL}}

\newcommand{\eqto}{\stackrel{\lower1.5pt\hbox{$\scriptstyle\sim\,$}}\to}

\begin{document}
\title[Equivariant unirationality]{Equivariant unirationality of tori in small dimensions}

\author[Yu. Tschinkel]{Yuri Tschinkel}
\address{
  Courant Institute,
  251 Mercer Street,
  New York, NY 10012, USA
}
\email{tschinkel@cims.nyu.edu}

\address{Simons Foundation\\
160 Fifth Avenue\\
New York, NY 10010\\
USA}

\author[Zh. Zhang]{Zhijia Zhang}
\address{
  Courant Institute,
  251 Mercer Street,
  New York, NY 10012, USA
}
\email{zz1753@nyu.edu}

\date{\today}

\begin{abstract}
%For an action of a finite group on a rational variety 
We study equivariant unirationality of actions of finite groups on tori of small dimensions.
\end{abstract}

\maketitle

\section{Introduction}
\label{sec.intro}

Rationality of tori over nonclosed fields is a well-established and active area of research,  going back to the work of Serre, Voskresenski\u i, Endo, Miyata, Colliot-Th\'el\`ene, Sansuc, Saltman,  Kunyavski\u i (classification of rational tori in dimension 3), and to more recent contributions of Lemire and Hoshi--Yamasaki (stably rational classification in dimensions $\le 5$), see, e.g., \cite{hoshiyamasaki} for a summary of results and extensive background material. 

In pursuing analogies between birational geometry over nonclosed fields and equivariant birational geometry, i.e., birational geometry over the classifying stack $BG$, where $G$ is a finite group, it is natural to consider algebraic tori in both contexts. While some of the invariants have a formally similar flavor, e.g., invariants of the geometric character lattice as a Galois, respectively, $G$-module, 
there are also striking differences. For example, a major open problem is to find examples of stably rational but nonrational tori over nonclosed fields. Over $BG$, there are examples already in dimension 2 \cite[Section 9]{lemire}. 
Furthermore, ``rational'' tori over $BG$ need not have $G$-fixed points!

To make this dictionary more precise: in the equivariant setup, one studies  
regular, but not necessarily generically free, actions of finite groups $G$ on smooth projective rational varieties $X$, over an algebraically closed field of characteristic zero. 
The following properties, the equivariant analogs of the notions of {\em (stable) rationality} and {\em unirationality}, have attracted attention: 

\begin{itemize}
    \item {\bf (L), (SL)} {\em lineariazability}, respectively, {\em stable linearizability}:
     there exists a linear representation $V$ of $G$ and a $G$-equivariant birational map 
     $$
     \bP(V)\dashrightarrow X, \quad \text{ respectively, } \quad \bP(V)\dashrightarrow X \times \bP^m,
     $$
     with trivial action on the $\bP^m$-factor, 
    \item {\bf (U)} {\em unirationality}: there exists a linear representation $V$ of $G$ and a $G$-equivariant dominant rational  map 
    $$
    \bP(V)\dashrightarrow X.$$
\end{itemize}
Property {\bf (U)} is also known as {\em very versality} of the $G$-action. It was explored in the context of essential dimension in, e.g., \cite{DR}. It 
has been studied for del Pezzo surfaces in \cite{Duncan}, and for toric varieties in \cite{Duncaness,KT-toric2}.

\ 

\noindent 
A necessary condition for both {\bf (SL)} and {\bf (U)} is 
\begin{itemize}
\item  {\bf(A)}:
for every abelian subgroup $A\subseteq G$ one has 
$
X^A\neq \emptyset.
$
\end{itemize}
A necessary condition for {\bf (SL)} is:
\begin{itemize}
\item {\bf (SP)}:
the Picard group $\Pic(X)$ is a stably permutation $G$-module. 
\end{itemize}
A necessary condition for {\bf (U)} is:
\begin{itemize} 
\item  {\bf(T)}: the action lifts to the universal torsor, see \cite[Section 5]{KT-toric2} and Section~\ref{sect:gene} for more details. 
\end{itemize}

These conditions are equivariant stable birational invariants of smooth projective varieties. 
We have
$$
{\bf (SL)}\quad  \Rightarrow \quad {\bf (U)},
$$
but the converse fails already for del Pezzo surfaces: there exist quartic del Pezzo surfaces satisfying {\bf (U)} but failing {\bf (SP)}, and thus {\bf (SL)}. 
By \cite[Theorem 1.4]{Duncan}, Condition {\bf (A)} is sufficient for {\bf (U)}, for regular, generically free actions 
on del Pezzo surfaces of degree $\ge 3$; same holds for 
smooth quadric threefolds, or intersections of two quadrics in $\bP^5$, by \cite{CTZ-uni}.
In \cite{KT-toric2} it was shown that regular, not necessarily generically free, actions on toric varieties are unirational, if and only if {\bf (T)} is satisfied. Using this, and \cite[Proposition 12]{HT-torsor}, we have, for $G$-actions on toric varieties arising from an injective homomorphism $G\hookrightarrow \Aut(T)$,
$$
{\rm \bf (U)}  + {\rm \bf (SP)}\Longleftrightarrow  {\rm \bf (SL)}.
$$

Our goal in this note is to obtain explicit, group-theoretic, criteria for 
stable linearizability and unirationality of actions of finite groups $G$ on 
smooth projective equivariant compactifications of tori $T=\bG_m^n$ in small dimensions. When $n=2$, 
Condition {\bf (A)} implies {\bf (U)} and {\bf (SL)} \cite{Duncan,HT-torsor}.  
Our main result is:

\begin{theo}
 \label{thm:main}
Let $T=\bG_m^3$, $G\subset \Aut(T)$ be a finite group, and 
$X$ a smooth projective $G$- and  $T$-equivariant compactification of $T$. Let 
$$
\pi^*: G\to \GL(\rN)
$$
be the induced representation on the cocharacter lattice $\rN$ of $T$.  
Assume that the $G$-action on $X$ satisfies Condition {\bf (A)}. Then 
\begin{itemize}
    \item the $G$-action is {\bf (U)} if and only if $\pi^*(G)$ does not contain, up to conjugation, the group
\begin{equation}
    \label{eqn:bad}
{\tiny \left\langle\begin{pmatrix}  0 &1 &-1\\
    1 &0 &-1\\
     0 & 0 & -1
   \end{pmatrix}, \begin{pmatrix}
       -1&  0 &0\\
    -1 &0 &1\\
 -1 & 1 & 0
\end{pmatrix}\right\rangle\simeq C_2^2,}
\end{equation}
\item 
the $G$-action is {\bf (SL)} if and only if $\pi^*(G)$ does not contain, up to conjugation, the group in \eqref{eqn:bad} 
or any of the groups
$$
{\tiny \left\langle
\begin{pmatrix}
    0&1&0\\
    0&0&1\\
    -1&-1&-1
\end{pmatrix},
\begin{pmatrix}
    -1&0&0\\
    0&-1&0\\
    0&0&-1
\end{pmatrix}\right\rangle\simeq C_2\times C_4,}
$$
$$
{\tiny \left\langle
\begin{pmatrix}
    0&0&1\\
    -1&-1&-1\\
    1&0&0
\end{pmatrix},
\begin{pmatrix}
    -1&-1&-1\\
    0&0&1\\
    0&1&0
\end{pmatrix},
\begin{pmatrix}
    -1&0&0\\
    0&-1&0\\
    0&0&-1
\end{pmatrix}\right\rangle\simeq C_2^3.}
$$
\end{itemize}
In particular, 
$$
{\bf (A)} + {\bf (SP)} \Longleftrightarrow {\bf (SL)}. 
$$

\end{theo}

Here is the roadmap of the paper: In Section~\ref{sect:gene} we recall basic toric geometry and group cohomology. In Section~\ref{sect:surf} we 
provide details on equivariant geometry of toric surfaces. In Section~\ref{sect:3-folds} we recall the construction of equivariant smooth projective models of 3-dimensional tori, following \cite{kunyavskii}. In Section~\ref{sect:uni} we prove the main technical lemmas needed for Theorem~\ref{thm:main}. A particularly difficult case, with $\pi^*(G)$ given by \eqref{eqn:bad},  
is outsourced to Section~\ref{sect:exept}.
In Section~\ref{sect:final} we summarize the main steps of the proof of Theorem~\ref{thm:main}.

\medskip

\noindent
{\bf Acknowledgments:} 
The first author was partially supported by NSF grant 2301983.

\section{Generalities}
\label{sect:gene}

We work over an algebraically closed field $k$ of characteristic zero.

\subsection*{Automorphisms}

Let $T=\bG_m^n$ be an algebraic torus over $k$.
The automorphisms of $T$ admit a description via the exact sequence
$$
1\to T(k)\to \Aut(T)\stackrel{\pi}{\lra} \GL(\rM) \to 1,
$$
where $\rM:=\mathfrak X^*(T)$ is the character lattice of $T$, which is dual to the cocharacter lattice $\rN$. For any finite subgroup $G\subset\Aut(T)$,
we have an exact sequence
\begin{equation} 
\label{eqn:exte}
1\to G_T \to G\stackrel{\pi}{\to} \bar{G}\to 1, \quad G_T:=T(k)\cap G. 
\end{equation}
The cases where $G$ fixes a point on $T$, without loss of generality $1\in T$, 
were studied in the context of geometry over nonclosed fields in \cite{kunyavskii, hoshiyamasaki}.
By way of contrast, here we allow the more general actions considered in \cite{HT-torsor}, where $X$ is an equivariant compactification of a {\em torsor} under a $G$-torus, these can still be (stably) linearizable. 

As a convention, the $G$-action is from the right throughout the paper. For example, choosing appropriate coordinates $\{ t_1,t_2,t_3\}$ on $T=\bG_m^3$, 
the matrices in \eqref{eqn:bad} correspond to actions on $T$ given by
$$
(t_1,t_2,t_3)\mapsto (t_2,t_1,\frac{1}{t_1t_2t_3}),\quad (t_1,t_2,t_3)\mapsto (\frac{1}{t_1t_2t_3},t_3,t_2).
$$

\subsection*{Smooth projective models}

To obtain a smooth, projective, $G$- and $T$-equivariant compactification $X$ of $T$, it suffices to choose a $\pi^*(G)$-invariant complete regular {\em fan} $\Sigma$ in the lattice of cocharacters $\rN$,
see, e.g., \cite[Section 1.3]{BT-aniso}.  Indeed, the translation action by $T(k)$ extends to any such compactification of $T$, by definition. Such a choice of $X$ and $\Sigma$ yields two exact sequences of $G$-modules
\begin{align}\label{eqn:funcseq}
    1\to k^\times \to k(T)^\times \to  \rM\to 0, 
\end{align}
and 
\begin{align}\label{eqn:picseq}
    0\to \rM\to \mathrm{PL}\to \mathrm{Pic}(X)\to 0,
\end{align}
where $\mathrm{PL}$ is a free $\bZ$-module with generators corresponding to 1-dimensional cones in $\Sigma$, or equivalently, irreducible components of $X\setminus T$.

\subsection*{Obstruction class}
The Yoneda product of the extensions
\eqref{eqn:funcseq} and 
\eqref{eqn:picseq} yields a cohomology class 
$$
\beta(X,G)\in\Ext^2(\Pic(X),k^\times)\simeq\rH^2(G,\Pic^\vee\otimes k^\times)\simeq \rH^3(G,\Pic(X)^\vee). 
$$
Effectively, $\beta(X,G)$ can be computed as the image of 
$$
\mathrm{id}_{\Pic(X)}\in \End(\Pic(X))^G
$$ 
under the composition of the following connecting homomorphisms, arising from tensoring \eqref{eqn:funcseq} and 
\eqref{eqn:picseq} by $\Pic(X)^\vee$:
\begin{equation}\label{eqn:compute}
\End(\Pic(X))^G\to\rH^1(G,\Pic(X)^\vee\otimes \rM)\to\rH^2(G,\Pic(X)^\vee\otimes k^\times).   
\end{equation}

As explained in \cite[Section 5]{KT-toric2}, if $Y\to X$ is a $G$-equivariant morphism then  
\begin{equation}
    \label{eqn:betay}
\beta(Y,G)=0 \quad \Rightarrow \quad \beta(X,G)=0. 
\end{equation}
For toric varieties, this is also a consequence of Theorem~\ref{thm:toricge} below.
By basic properties of cohomology, we observe: 

\begin{lemm} 
\label{lemm:beta}
We have
$$
\beta(X,G)=0 \,\,\Longleftrightarrow \,\, \beta(X,G_p)=0 \,\,\text{ for all $p$-Sylow subgroups } G_p\subseteq G.
$$
\end{lemm}

\begin{proof}
Since $G$ is finite, for any $G$-module $\rP$, the sum of restriction homomorphisms gives
an embedding 
$$
\rH^2(G,\rP)\to\oplus_{p} \, \rH^2(G_p,\rP),
$$
where $p$ runs over primes dividing $|G|$.
\end{proof}

\subsection*{Bogomolov multiplier}

As explained in \cite[Section 3.6]{HT-torsor}, functoriality implies that 
\begin{equation} 
\label{eqn:xg}
X^G\neq \emptyset \quad  \Rightarrow  \quad 
\beta(X,G)=0. 
\end{equation}
Thus Condition {\bf (A)} forces that 
\begin{equation}\label{eqn:bogoA}
  \beta(X,G)\in \mathrm{B}^3(G,\Pic(X)^\vee),  
\end{equation}
where for any $G$-module $\rP$ and $n\in \bN$, we put
$$
\mathrm{B}^n(G,\rP):=\bigcap_{A} \Ker\left(\rH^n(G,\rP) \stackrel{\res}{\lra} \rH^n(A,\rP)\right), 
$$
the intersection over all abelian subgroups $A\subseteq G$; this 
is the generalization of the Bogomolov multiplier 
$$
\rB^2(G,k^\times),
$$
where the $G$-action on $k^\times$ is trivial, considered in \cite[Section 2]{tz}. 
 Note that while the vanishing of $\beta(X,G)$ is a stable birational invariant, the group $\mathrm{B}^n(G,\Pic(X)^\vee)$ is not, in general, as the following example shows.

\begin{exam}
\label{exam:change}
Let  $G$ be a group with a nontrivial Bogomolov multiplier. Let $V$ be a faithful linear representation of $G$ and 
$X=\bP( {\bf 1}\oplus V)$. Let $\tilde{X}$ be the blowup of $X$ in the $G$-fixed point. 
Then 
 $$
\rB^3(G,\Pic(\tilde{X})^\vee)\neq \rB^3(G,\Pic(X)^\vee)\oplus  
\rB^3(G,\bZ).  
$$
\end{exam}

We will need the following technical statement about generalized Bogomolov multipliers: 

\begin{lemm}
    \label{lemm:bog}
 Let $H\subset G$ be a normal subgroup with cyclic quotient $G/H=C_m$. Let $\rP_0$ be an $H$-module and $  \rP=\oplus_{j=1}^m \rP_0$ the induced $G$-module. Then the restriction homomorphism 
 $$
 \rH^2(G, \rP)\to  \rH^2(H, \rP)
 $$
is injective. In particular, we have
$$
\rB^2(H,\rP_0)=0 \quad \Rightarrow \quad \rB^2(G,\rP)=0.  
$$
\end{lemm}

\begin{proof}
The Hochschild–Serre spectral sequence yields:
    \begin{align*}
        0\to\rH^1(G/H,\rP^H)\to  \rH^1(G,\rP)\to \rH^1(H,\rP)^{G/H}\to\rH^2(G/H,\rP^H)\to \\
\to \ker\left(\rH^2(G, \rP) \to  \rH^2(H, \rP) \right) 
        %\rK 
        \to \rH^1(G/H,\rH^1(H,\rP))   %\to\Ker(\rH^3(G/H,\rP^H)\to\rH^3(G,\rP)), 
    \end{align*}
%with 
%$$
%\rK=\ker\left(\rH^2(G, \rP)\to  \rH^2(H, \rP) \right). 
%$$
We have
$$
\rH^2(G/H,\rP^H) = \rH^1(G/H,\rH^1(H,\rP)) =0.
$$
Indeed, $G/H$-acts via cyclic permutations on the summands of $\rP^H$ and $\rH^1(H,\rP)$; cohomology of cyclic groups acting via cyclic permutations vanishes in all degrees $\ge 1$. 
It follows that 
$$
\ker\left(\rH^2(G, \rP) \to  \rH^2(H, \rP) \right) =0.
$$

%Consider the diagram
%\centerline{
%\xymatrix{
%\rH^2(G, \rP) \ar[r] \rH^2(H, \rP)=\oplus_{j=1}^m \rH^2(H,\rP) \ar[d]   \\
%& \oplus_{j=1}^\oplus_{A}  \rH^2(A,\rP) 
%}
%5}
%\noindent
%where $A$ runs over abelian subgroups of $H\subset G$.
Thus, a nonzero class $\alpha\in \rH^2(G, \rP)$ remains nonzero in 
$$
\rH^2(H,\rP)=\oplus_{j=1}^m \rH^2(H,\rP_0).
$$
If $\alpha \in \rB^2(G,\rP)$, 
then the restriction of $\alpha$ to $H$ lies in $\rB^2(H, \rP)$, contradicting the assumption that $\rB^2(H, \rP_0)=0$.
\end{proof}

\subsection*{Geometric applications}

%Let $X$ be a smooth projective compactification of a torus $G$ as above.
%If $\bar{G}$ is trivial then the $G$-action has fixed points in the boundary $X\setminus T$, and thus 
%$X$ satisfies {\bf (A)}, {\bf (T)}, and {\bf (U)}. 
The following theorem characterizes unirationality of $G$-actions on toric varieties: 

\begin{theo}  
\label{thm:toricge}
\cite[Section 4]{HT-torsor}, \cite[Section 5]{KT-toric2}
Let $X$ be a smooth projective $T$-equivariant compactification of a torus $T$, 
with an action of a finite group $G$ arising from a homomorphism $\rho: G\to \Aut(T)$. 
Then 
$$
\beta(X,G)=0 
\Longleftrightarrow {\rm \bf (T)} \Longleftrightarrow  {\rm \bf (U)}. 
$$  
\end{theo}

A related result, concerning {\em versality} of generically free actions on toric varieties is 
\cite[Theorem 3.2]{Duncaness}:
it is equivalent to the vanishing of $\beta(X,G)$, in our terminology. 
In particular, for such actions, versality is equivalent to {\em very versality} (i.e., unirationality), see
\cite{DR} for further details regarding these notions. 
Here, we allow actions with nontrivial generic stabilizers, 
i.e., when the kernel of 
$\rho$ is nontrivial. 

A consequence of \cite[Proposition 12]{HT-torsor} is:

\begin{theo}
    \label{thm:torict}
Let $X$ be a smooth projective $T$-equivariant compactification of a torus $T$, 
with a regular action of a finite group $G$ arising from an injective homomorphism $\rho: G\hookrightarrow \Aut(T)$. 
Then
$$
{\rm \bf (T)}  + {\rm \bf (SP)}\Longleftrightarrow  {\rm \bf (SL)}.
$$
\end{theo}

Note that 
$$
{\bf (A)} \not \Rightarrow {\bf (U)},
$$
even for generically free actions on toric threefolds, see Section~\ref{sect:3-folds}. In fact, not even for $X=\bP^1$, if we allow generic stabilizers! 

\begin{exam}
\label{exam:p1}
Consider $X=\bP^1$. If $G\subset \Aut(T)$, i.e., 
$$
G\subset \bG_m(k)\rtimes C_2,
$$
then $G$ is either cyclic or dihedral. Actions of cyclic groups and of dihedral groups 
$\fD_n$ of order $2n$, with $n$ odd,  are unirational; actions of $\fD_n$, with $n$ even, are not unirational, since they contain the subgroup $C_2^2$ which has no fixed points on $\bP^1$, i.e., failing Condition {\bf (A)}. In particular, for generically free actions we have
$\mathrm{\bf{(A)}}\Leftrightarrow \mathrm{\bf{(U)}}$.

For non-generically free $G$-actions on $\bP^1$, considered in \cite[Example 2.2]{KT-toric2}, Condition {\bf  (A)} does not suffice to characterize unirationality. By \eqref{eqn:bogoA}, such groups must have a nontrivial 
Bogomolov multiplier 
$$
\rB^2(G,k^\times)\simeq \rB^3(G,\bZ).
$$
For example, let $G$ be the group of order 64, with {\tt GAP ID (64,149)}; 
this is the smallest group with nontrivial $\rB^2(G,k^\times)$. 
There is a unique subgroup $H\simeq C_2\times \fQ_8\subset G$, with {\tt GAP ID (16,12)}, and  $G/H\simeq C_2^2$. 
Consider a homomorphism $\rho:G\to\PGL_2(k)$ with kernel $H$ and image $C_2^2$. 
The resulting $G$-action on $\bP^1$ is not generically free, but satisfies Condition {\bf(A)} -- no abelian subgroup of $G$ surjects onto $C_2^2$, via $\rho$. We compute as in \eqref{eqn:compute} that
$$
0\ne\beta(\bP^1,G)\in\rB^2(G,k^\times)\subset \rH^2(G,k^\times)=\rH^2(G,\Pic(\bP^1)^\vee\otimes k^\times). 
$$
Another way to see this is to observe the equality of commutator subgroups
$$
[G,G]=[G,H],
$$
which is impossible if the extension of $G$ by $k^\times$ associated with the class $\beta(\bP^1,G)$ splits, see \cite[Example 2.2]{KT-toric2}. Thus, {\bf (U)} fails for this action.
\end{exam}

\subsection*{Group cohomology}

The computation of the obstruction class $\beta(X,G)$ 
relies on an explicit resolution of the group ring. We write down such resolutions for groups that will be relevant for the analysis of 3-dimensional toric varieties in Section~\ref{sect:3-folds}. By convention, we work with right $G$-modules, i.e., the group $G$ acts from the right.

$\bullet$ Let $G=\fQ_{2^n}$ be the generalized quaternion group of order $2^n$, with a presentation
$$
\fQ_{2^n}:=\langle x,y\vert x^{2^{n-2}}=y^2,\,\, xyx=y\rangle,
$$
and $\rP$ a $G$-module. By \cite[\S XII.7]{cartaneilenberg}, the cohomology groups $\rH^i(G,\rP)$, $i=0,1,2,3$, can be computed as the $i$-th cohomology of the complex

\begin{equation}\label{eqn:resolutionQ8}
 \rP \stackrel{\tiny\begin{pmatrix}
  1-x&  1-y \end{pmatrix}}{\xrightarrow{\hspace*{2cm}} }
  \rP^2 \stackrel{\tiny\begin{pmatrix}
  N_x&  {yx}+1\\
  -1-y&x-1
  \end{pmatrix}}{\xrightarrow{\hspace*{2.3cm}}}
\rP^2\stackrel{\tiny\begin{pmatrix}
 1-x \\yx-1 \end{pmatrix}}{\xrightarrow{\hspace*{0.8cm}}}
\rP \stackrel{\sum_{g\in G}g}{\xrightarrow{\hspace*{1cm}}} \rP
\cdots
\end{equation}
where 
$$N_x=1+x+x^2+\cdots+x^{(2^{n-2}-1)}. 
$$
%Recall that cohomology groups of generalized quaternion groups are 4-periodic.  
%The resolution \eqref{eqn:resolutionQ8} continues periodically.

\

$\bullet$ Let $G=\fD_{2^{n-1}}$ be the dihedral group of order $2^n$, with a presentation 
$$
\fD_{2^{n-1}}:=\langle x,y\vert x^{2^{n-1}}=y^2=yxyx=1\rangle,
$$
and $\rP$ a $G$-module. By \cite[\S IV.2]{ademmilgram}, the cohomology groups $\rH^i(\fD_n,\rP)$, $i=0,1,2$, can be computed as the $i$-th cohomology of

\begin{equation}\label{eqn:resolutiondihedral}
 \rP \stackrel{\tiny\begin{pmatrix}
  1-x& \!\!\! 1-y \end{pmatrix}}{\xrightarrow{\hspace*{1.5cm}} }
  \rP^2 \stackrel{\tiny\begin{pmatrix}
  N_x& \!\! 1+{yx}&\!\!0\\
  0&\!\!x-1&\!\!1+y
  \end{pmatrix}}{\xrightarrow{\hspace*{1.8cm}}}
\rP^3\stackrel{\tiny\begin{pmatrix}
  1-x& \!\! 1+y&\!\!0&\!\!0\\ 
  0& \!\! -N_{x}&\!\!1-{yx}&\!\!0\\
   0& \!\! 0&\!\!1-x&\!\!1-y   \end{pmatrix}}{\xrightarrow{\hspace*{3.2cm}}}
\rP^4 
\cdots
\end{equation}
where
$$
N_x=1+x+x^2+\cdots+x^{(2^{n-1}-1)}.
$$

$\bullet$ Let $G=\mathfrak{SD}_{2^n}$ be the semidihedral group of order $2^n$, with presentation 
$$
\mathfrak{SD}_{2^n}:=\langle x,y\vert x^{2^{n-1}}=y^2=1,\,\, yxy=x^{2^{n-2}-1}\rangle,
$$
and $\rP$ a $G$-module. Using the resolution constructed in \cite{generalov}, we can compute $\rH^i(G,\rP)$, $i=0,1,2$,  as the $i$-th cohomology of 

\begin{equation}\label{eqn:resolutionsemidihedral}
 \rP \stackrel{\tiny\begin{pmatrix}
  1-x& 1-y
  \end{pmatrix}}{\xrightarrow{\hspace*{2cm}}}
  \rP^2\stackrel{\tiny\begin{pmatrix}
  L_2& 0\\
  L_1& 1+y
  \end{pmatrix}}{\xrightarrow{\hspace*{2.3cm}}}
\rP^2\stackrel{\tiny\begin{pmatrix}
  -L_3& 0\\
  L_4&1-y
  \end{pmatrix}}{\xrightarrow{\hspace*{2.3cm}}}
\rP^2 
\cdots
\end{equation}
where 
 $$
 L_1=x^{2^{n-3}+1}-1,\quad
 L_2=\sum_{r=0}^{2^{n-3}}x^r-\left(\sum_{r=0}^{2^{n-3}-2}x^r\right)\cdot y,
 $$
 $$
 L_3=(x^{2^{n-3}-1}-1)(1+y),\quad L_4=(x^{2^{n-3}+1}-1)(x^{2^{n-3}-1}-1).
 $$

\section{Toric surfaces}
\label{sect:surf}

In this section, we recall the classification of unirational and linearizable actions of subgroups $G\subset \Aut(T)$ on smooth projective toric surfaces $X\supset T$. 
The maximal finite subgroups of $\GL_2(\bZ)$ are
$$
\fD_4 \quad\text{and}\quad \fD_6.
$$
The recipe of Section~\ref{sect:gene} shows that all actions as above can be realized as regular actions on $X=\bP^1\times \bP^1$ and respectively, $X=\mathrm{dP}_6$, the del Pezzo surface of degree 6.

\begin{prop} 
\label{prop:dp6}
    Let $X$ be a smooth projective toric surface with an action of a finite group $G\subset \Aut(T)$. Then 
    $$
{\bf (SL)}\Longleftrightarrow {\bf (U)} \Longleftrightarrow {\bf (A)}.
    $$
\end{prop}

\begin{proof}
The right equivalence is proved in \cite{Duncan}. The left equivalence follows from the general result
Theorem~\ref{thm:torict}, i.e., \cite[Proposition 12]{HT-torsor}: if the generically free $G$-action satisfies {\bf (T)} and $\Pic(X)$ is a stably permutation $G$-module, then the action is stably linearizable.
The stable permutation property of the
$G$-action on $\Pic(X)$ is clear for $X=\bP^1\times \bP^1$; for $X=\mathrm{dP}6$, see \cite[Section 6]{HT-torsor}.   
\end{proof}

It will be convenient to choose coordinates $\{t_1,t_2\}$ of $T=\bG_m^2$. This determines a basis $\{m_1,m_2\}$ of the lattice $\rM$ and $\{n_1,n_2\}$ of its dual $\rN$. Assume that in the basis $\{n_1,n_2\}$, the $\fD_4$-action is generated by the involutions
$$
\iota_1:={\tiny\begin{pmatrix}
         -1&0\\0&-1
     \end{pmatrix},}\quad\iota_2:={\tiny\begin{pmatrix}
         -1&0\\0&1
     \end{pmatrix},}\quad \iota_3:={\tiny\begin{pmatrix}
         0&1\\1&0
     \end{pmatrix}.}
     $$
Consider the  involutions in $T(k)\cap G$, in the coordinates $\{t_1,t_2\}$,
$$
\tau_1:=(-1,-1), \quad \tau_2:=(-1,1), \quad \tau_3:=(1,-1).
$$
A classification of $G$ yielding versal (and thus unirational) actions can be found in \cite[Section 4.1]{Duncaness}. Note that only 2 and 3-groups are relevant. 
\begin{prop}
    Let $G\subset \Aut(T)$ be a finite subgroup acting on a smooth projective toric surface $X\supset T$. Then \begin{itemize}
        \item if $G$ is a 3-group, then the $G$-action is unirational if and only if $\pi^*(G)=1$ or $G_T=1$.  
         \item if $G$ is a 2-group, then the $G$-action is unirational if and only if one of the following holds, up to conjugation,
         \begin{itemize}
          \item $\pi^*(G)=1$ or $\langle\iota_3\rangle$,
             \item $\pi^*(G)=\langle\iota_2\rangle$ and $G_T\subset \{(1,t):t\in\bG_m(k)\}\subset T(k)$, 
             \item $G_T=1$ otherwise.
         \end{itemize}
    \end{itemize}
Moreover, for actions of $p$-groups, unirationality implies linearizability.     
\end{prop}

Here, we  recall the classification on linearizable $G$-actions on toric surfaces, with $G\subset \Aut(T)$, following \cite{DI} and \cite{PSY}. Consider the case $\pi^*(G)\subseteq \fD_4$: 
\begin{itemize}
    \item If $\rk\Pic(X)^G=1$ then the $G$-action is linearizable if and only if $\pi^*(G)$ is conjugate to $\langle \iota_3\rangle \simeq C_2$ in $\GL_3(\bZ)$.
    \item If $\rk\Pic(X)^G=2$ then the $G$-action is linearizable if and only if up to conjugation one of the following holds:
    \begin{itemize}
        \item $\pi^*(G)=\langle \iota_2\rangle$ and  $$
    G_T\subset\{(t_1,t_2):t_1,t_2\in\bG_m(k), \quad\mathrm{ord}(t_1) \text{ is odd}\}.
    $$
    \item $\pi^*(G)=\langle \iota_1\rangle$ or $\langle \iota_1,\iota_2\rangle$ and    $|G_T|$ is odd.
      \end{itemize}
\end{itemize}

We turn to $\pi^*(G)\subseteq \fD_6$:
\begin{itemize}
\item If $\rk\Pic(X)^G=1$ then the $G$-action is linearizable if and only if  $G_T=1$ and $G\simeq C_6$ or $G\simeq\fS_3$.
\item If $\rk\Pic(X)^G=2$ and  $\pi^*({G})=C_3$ or $\fS_3$, then the action is linearizable if and only if $3\nmid |G_T|$; all other possibilities for $\pi^*(G)$ are realized as subgroups of $\fD_4$, covered above. 
\end{itemize}

\section{Toric threefolds: smooth projective models} 
\label{sect:3-folds}

We start with the classification of actions and their realizations on smooth projective toric threefolds, following \cite{kunyavskii}. Let $G\subset \Aut(T)$ be a finite subgroup where $T=\bG_m^3$. Recall from the exact sequence \eqref{eqn:exte} that $\bar G=\pi(G)$ is a subgroup of $\GL_3(\bZ)$.
There are two isomorphism classes
of maximal finite subgroups of $\GL_3(\bZ)$: 
$$
C_2\times\fS_4 \quad \text{ and } \quad C_2\times \fD_6. 
$$
The first group gives {\em three} conjugacy classes in $\GL_3(\bZ)$, 
referred to as Case {\bf (C)}, {\bf (S)}, and {\bf (P)}, respectively. 
The other group gives one conjugacy class, called Case {\bf (F)}.

\subsection*{Case (C)} 
Here, $X=(\bP^1)^3$, with $\bar{G}\subset C_2\times \fS_4$ and the action visible from the presentation  
$$
1\to C_2^3\to C_2\times \fS_4\to \fS_3\to 1,
$$
with $\fS_3$ permuting the factors and $C_2$ acting as an involution on the corresponding $\bP^1$.

\

\subsection*{Case (F)} 
In this case, $X=\bP^1\times \mathrm{dP6}$, and $\bar{G}\subset C_2\times \fD_6$, with 
$C_2$ acting via the standard involution on $\bP^1$, and $\fD_6$ acting on dP6 as described in Section~\ref{sect:surf}.

\

\subsection*{Case (P)} In this case, $X$ is the blowup of 
$$
\{u_1u_2u_3u_4=v_1v_2v_3v_4\}\subset\bP^1_{u_1,v_1}\times\bP^1_{u_2,v_2}\times\bP^1_{u_3,v_3}\times\bP^1_{u_4,v_4}
$$
in its 6 singular points, and $\bar G\subset C_2\times\fS_4$. The corresponding $\pi^*(G)$-invariant fan $\Sigma$ consists of 99 cones: 32 three-dimensional cones, 48 two-dimensional cones, 18 rays, and the origin. We have 
$$
\Pic(X)=\bZ^{15}.
$$

\

\subsection*{Case (S)} Here,  $X$ is the blowup of 
$\bP^3$ in 4 points and the 6 lines through these points, with $\bar{G}=C_2\times \fS_4$, 
acting via permutations on the 4 points and 6 lines, with $C_2$ corresponding to the Cremona involution on $\bP^3$, which is regular on $X$. A singular model is the intersection of two quadrics
$$
\{y_1y_4-y_2y_5=y_1y_4-y_3y_6 \} \subset \bP^5_{y_1,y_2,y_3,y_4,y_5}.
$$
Blowing up its 6 singular points one obtains $X$, see \cite[Section 9]{HT-Q} for an extensive discussion of this geometry. The corresponding fan $\Sigma$ consists of 75 cones: 24 three-dimensional cones, 36 two-dimensional cones, 14 rays, and the origin. We have 
$$
\Pic(X)=\bZ^{11}.
$$

\

By Lemma~\ref{lemm:beta}, to establish property {\bf (U)} it suffices to consider $p$-Sylow subgroups of $G$, 
in our case, $p=3$ or $2$.

\subsection*{Models for 3-groups}

The are two finite 3-subgroups of $\GL_3(\bZ)$, both isomorphic to $C_3$. They are generated by 
$$
{\tiny \begin{pmatrix}
 0&1&0\\
 0&0&1\\
 1&0&0
\end{pmatrix},}
\quad \text{respectively,}\quad
{\tiny \begin{pmatrix}
 1&0&0\\
 0&-1&-1\\
 0&1&0
\end{pmatrix}.}
$$
The corresponding model $X$ can be chosen to be $\bP^3$ and $\bP^1\times \bP^2$ respectively.

\subsection*{Models for 2-groups}

There are three conjugacy classes of $C_2\times \fS_4$ in $\GL_3(\bZ)$, but only {\em two} conjugacy classes of their 2-Sylow subgroups,
generated respectively by 
$$
\left\langle{\tiny \begin{pmatrix}
            0&0&1\\
            0&1&0\\
        -1&0&0
        \end{pmatrix},\begin{pmatrix}
            -1&0&0\\
            0&-1&0\\
        0&0&-1
        \end{pmatrix},\begin{pmatrix}
            -1&0&0\\
            0&1&0\\
        0&0&1

        \end{pmatrix}}\right\rangle\simeq C_2\times \fD_4,
$$
and
$$
\left\langle{\tiny \begin{pmatrix}
            0&0&1\\
            -1&-1&-1\\
        1&0&0
        \end{pmatrix},\begin{pmatrix}
            -1&0&0\\
            0&-1&0\\
        0&0&-1
        \end{pmatrix},\begin{pmatrix}
            1&0&0\\
            0&1&0\\
        -1&-1&-1
        \end{pmatrix}}\right\rangle\simeq C_2\times \fD_4.
$$
The first group is realized on $X=(\bP^1)^3$, and the other on either {\bf (S)} or {\bf (P)} model. 

In the analysis below, we need a simpler smooth projective model when $\pi^*(G)$ is contained in the $\fD_4\subset\GL_3(\bZ)$ generated by  
$$
{\tiny\tau_1:=\begin{pmatrix}
            1&0&0\\
            -1&-1&-1\\
            0&0&1
        \end{pmatrix},
        \quad
        \tau_2:=\begin{pmatrix}
            1&1&1\\
            -1&0&0\\
            0&-1&0     
    \end{pmatrix}.}
    $$
Let $\Sigma$ be the fan in $\rN$ generated by 6 rays with generators
$$
\begin{tabular}{lll}
    $v_1=(-1,0,-1)$,&$v_2=( 0, -1,  0)$,&
    $v_3=( 0,  0,  1)$,\\
    $v_4=( 1,  0,  0)$,
    &$v_5=( 1,  1,  1)$,
    &$v_6=( 1,  0,  1)$.
\end{tabular}
$$
and 8 cones 
$$
S_1=  \langle v_1, v_4, v_5 \rangle,
S_2= \langle v_1, v_3, v_5 \rangle,
S_3=\langle v_1, v_2, v_3 \rangle,
S_4=\langle v_1, v_2, v_4 \rangle,
$$
$$
S_5=\langle v_4, v_5, v_6 \rangle,
S_6=\langle v_3, v_5, v_6 \rangle,
S_7=\langle v_2, v_3, v_6 \rangle,
S_8=\langle v_2, v_4, v_6 \rangle.
$$
Then $\Sigma$ is $\pi^*(G)$-invariant and the toric variety $X=X(\Sigma)$ is the blowup of a cone over a smooth quadric surface at its vertex.

\section{Toric threefolds: unirationality}
\label{sect:uni}

Let $T=\bG_m^3$ and $G\subset \Aut(T)$ be a finite group, acting on a smooth projective $X$, which is
a $G$- and $T$-equivariant compactification of $T$. 
We recall the exact sequence
$$
1\to G_T\to G\stackrel{\pi}{\lra} \bar{G}\to 1.
$$
In this section, we classify unirational $G$-actions, in particular, these satisfy Condition {\bf (A)}.

\begin{prop}\label{prop:dim33group}
Let $G\subset \Aut(T)$ be 3-group such that the $G$-action on $X$ satisfies Condition {\bf (A)}. Then 
it satisfies {\bf (U)}.
\end{prop}

\begin{proof} 
As explained above, the model $X$ can be chosen to be either $\bP^3$ or $\bP^1\times \bP^2$.
In the first case, the action is linear. 
When $X=\bP^1\times \bP^2$, the $G$-action on $\Pic(X)$ is trivial
and 
$$
\beta(X,G)\in \rH^2(G,k^\times )\oplus \rH^2(G,k^\times );
$$
by \cite[Remark 5.5]{KT-toric2}, obstruction to $\beta(X,G)=0$ equals the Amitsur obstruction for each factor. We have an extension
$$
1\to G_T\to G\to C_3\to 1,
$$
with $G_T$ abelian; the Bogomolov multiplier $\rB^2(G,k^\times)=0$, by, e.g., \cite[Lemma 3.1]{KT-Brauer}.  Condition {\bf (A)} implies that the $G$-action lifts
to a linear action on each factor. 
\end{proof}

\begin{prop}\label{prop:dim32group}
Let $G\subset \Aut(T)$ be 2-group such that the $G$-action on $X$ satisfies Condition {\bf (A)}. Then 
it satisfies {\bf (U)} if and only if one of the following holds:
\begin{itemize}
    \item $G_T=1$, or
    \item $G_T\neq 1$ and $\pi^*(G)$ is not conjugated to 
\begin{equation} 
\label{eqn:C22}
 \fK_9=
 {\tiny\left\langle\begin{pmatrix}  0 &1 &-1\\
    1 &0 &-1\\
     0 & 0 & -1
   \end{pmatrix}, \begin{pmatrix}
       -1&  0 &0\\
    -1 &0 &1\\
 -1 & 1 & 0
\end{pmatrix}\right\rangle}\simeq C_2^2.
\end{equation}
\end{itemize}
\end{prop}
\begin{proof}
    The assertion follows from Lemmas~\ref{lemm:indeed1},~\ref{lemm:exclude},~\ref{lemm:indeed},~\ref{lemm:indeed4},~\ref{lemm:c22}, and~\ref{lemm:d4sum}.
\end{proof}

The rest of this section is devoted to the proof of this proposition. There are 2 conjugacy classes of maximal 2-groups in $\GL_3(\bZ)$, both isomorphic to 
$$
C_2\times \fD_4.
$$
The corresponding toric models are {\bf (C)}, and {\bf (S)} or {\bf (P)}. We proceed with a case-by-case analysis of actions;  altogether, we have to consider 36 conjugacy classes of finite subgroups $\pi^*(G)\subset\GL(\rN)$. We summarize:

 \

\begin{itemize}
    \item When $G_T=1$: {\bf (U)} holds, by Lemma~\ref{lemm:indeed1}. 
\item When $G_T\neq 1$ and $\pi^*(G)$ isomorphic to
\begin{itemize}
    \item $C_2, C_4$: {\bf (U)} holds, by Lemmas~\ref{lemm:indeed} and \ref{lemm:indeed4}. 
    \item $C_2^2$: see Lemma~\ref{lemm:c22}. 
    \item $\fD_4$: see Lemma~\ref{lemm:d4}. 
    \item $C_2^3$,  $C_2\times C_4$, $C_2\times\fD_4$: all such actions fail Condition {\bf (A)}, by Lemma~\ref{lemm:exclude}.  
\end{itemize}

\end{itemize}

\begin{lemm}
\label{lemm:indeed1}
Assume that $G\subset \Aut(T)$ is a 2-group with $G_T=1$.  
Then 
$$
\text{{\bf(A)}} \Longleftrightarrow \text{{\bf(U)}}. 
$$
\end{lemm}    

\begin{proof}
When $G$ is abelian, the claim follows from \eqref{eqn:xg}.
It remains to consider the cases when $G=\fD_4$ or $C_2\times \fD_4$. 
Via {\tt HAP}, we have computed that for corresponding models $X$, i.e.,  {\bf (C)} or {\bf (P)}, 
the generalized Bogomolov multiplier satisfies
$$
\rB^2(G,\Pic(X)^\vee\otimes k^\times)=0.
$$
 Condition {\bf (A)} implies that 
$$
\beta(X,G)\in \rB^2(G,\Pic(X)^\vee\otimes k^\times),
$$
and thus $\beta(X,G)=0$; it remains to apply
Theorem~\ref{thm:toricge}.
\end{proof}

Note that $X$ may fail to have $G$-fixed points even when $G_T=1$, see \cite[Remark 5.2]{CTZ-burk}. 

\

From now on,   we assume that
\begin{itemize}
    \item 
$G\subset \Aut(T)$ is a 2-group and
\item $G_T\neq 1$.
\end{itemize}

\begin{lemm} 
\label{lemm:exclude}
Assume that $\pi^*(G)$ 
contains
$$
\eta:=\mathrm{diag}(-1,-1,-1)\in\GL(\rN).
$$
Then the $G$-action fails Condition {\bf(A)}. 
\end{lemm}

\begin{proof} 
Indeed, $\eta$ can be realized as the diagonal involution on $(\bP^1)^3$, and  any translation by a 2-group will produce a $C_2^2$ action without fixed points. 
\end{proof}

\ 

After excluding groups containing $\eta$, it suffices to consider
$$
\pi^*(G)=
C_2, \quad C_4, \quad C_2^2, \quad \fD_4.  
$$

\begin{lemm}
\label{lemm:indeed}
Assume that $\pi^*(G)=C_2$. 
Then 
$$
\text{{\bf(A)}} \Longleftrightarrow \text{{\bf(U)}}. 
$$
\end{lemm}    

\begin{proof}  
Apart from $\langle\eta\rangle$, there are 4 conjugacy classes of groups of order 2 in $\GL_3(\bZ)$, generated by 
$$
{\tiny
\iota_1\!=\!\begin{pmatrix}
    1&0&0\\
    0&-1&0\\
    0&0&-1
\end{pmatrix},\,\,\iota_2\!=\!\begin{pmatrix}
    1&0&0\\
    0&1&0\\
    0&0&-1
\end{pmatrix},\,\,\iota_3\!=\!\begin{pmatrix}
    0&1&0\\
    1&0&0\\
    0&0&1
\end{pmatrix},\,\, \iota_4\!=\!\begin{pmatrix}
    0&1&0\\
    1&0&0\\
    0&0&-1
\end{pmatrix}.
}
$$
%\begin{enumerate}
   % \item $\bZ[-1]\oplus \bZ[-1]\oplus \bZ$
  %  \item $\bZ\oplus \bZ \oplus \bZ[-1]$, 
   % \item $\bZ[C_2]\oplus \bZ$, 
   % \item $\bZ[C_2]\oplus \bZ[-1]$.
%\end{enumerate}
The first case is realized on $(\bP^1)^3$ and the last three cases in $\bP^2\times \bP^1$. 
 By \cite[Remark 5.2]{KT-toric2}, unirationality is determined by unirationality of all of the $\bP^1$ and $\bP^2$ factors, which is equivalent to triviality of the Amitsur invariant, see Example~\ref{exam:p1}.  

Since $G$ is an extension of the cyclic group $C_2$ by an abelian group $G_T$, the Bogomolov multiplier $\rB^2(G,k^\times)=0$. Together with Condition {\bf (A)} this implies that the Amitsur invariant for the action on each factor is trivial, and the $G$-action on $X$ is unirational. 
\end{proof}

\begin{rema}
    \label{rema:alt}
Alternatively, one can check that for $\pi^*(G)=C_2$,  
the $G$-action satisfies Condition {\bf (A)} if and only if one of the following holds:
\begin{itemize}
    \item $\pi^*(G)=\langle\iota_1\rangle$, and $G_T\subset \{(t,1,1):t\in\bG_m(k)\}
\subset T(k)$.
    \item  $\pi^*(G)=\langle\iota_2\rangle$, and $G_T\subset \{(t_1,t_2,1):t_1,t_2\in\bG_m(k)\}
\subset T(k)$.
     \item  $\pi^*(G)=\langle\iota_3\rangle$, and $G_T$ is any subgroup of $T(k)$.
    \item  $\pi^*(G)=\langle\iota_4\rangle$, and $(t,t,-1)\not\in G_T$ for any $t\in \bG_m(k)$.
\end{itemize}
Using this description, we see that Condition {\bf (A)} is also equivalent to  $X^G\ne \emptyset$, in the first three cases.
\end{rema}

\begin{lemm} 
\label{lemm:indeed4}
Assume that $\pi^*(G)=C_4$. 
Then 
$$
{\bf (A)} \Longleftrightarrow  {\bf (U)}. 
$$
\end{lemm}

\begin{proof}
There are 4 conjugacy classes of $C_4\subset \GL_3(\bZ)$, generated by 
$$
{\tiny \theta_1\!=\!\begin{pmatrix}
    1&0&0\\
    0&0&1\\
    0&-1&0
\end{pmatrix},\theta_2\!=\!\begin{pmatrix}
    -1&0&0\\
    0&0&1\\
    0&-1&0
\end{pmatrix},
\theta_3\!=\!\begin{pmatrix}
    -1&-1&-1\\
    1&0&0\\
    0&1&0
\end{pmatrix},
\theta_4\!=\!\begin{pmatrix}
    1&1&1\\
    -1&0&0\\
    0&-1&0
\end{pmatrix}}.
$$
The first two cases are realized on $\bP^1\times Q$, where $Q=\bP^1\times\bP^1$. The third case on $\bP^3$. The fourth can be realized on either the {\bf (P)} or {\bf (S)} model.

\ 

\noindent
{\bf Case $\theta_1$}: Note that $\theta_1^2=\iota_1$. Condition {\bf (A)} implies that
    $$
    G_T\subset \{(t,1,1):t\in\bG_m(k)\}\subset T(k)
    $$
    and $G$ fixes a point on $X=\bP^1\times Q$; therefore, the $G$-action satisfies {\bf (U)}. 
    %It follows that the $G$-action satisfies {\bf (U)} if  and only it satisfies{\bf (A)}.

\

\noindent 
{\bf Case $\theta_2$}: We also have $\theta_2^2=\iota_1$. Condition {\bf (A)} implies that $G_T$ contains
    $$
   \iota=(-1,1,1)\in T(k).
    $$
  However, for $g\in G$ such that $\pi^*(g)=\theta_2$, the abelian subgroup $\langle g,\iota\rangle$ of $G$ does not fix points on $X$, contradiction. 
    %It follows that the $G$-action satisfies {\bf (A)} if and only it satisfies {\bf (U)}.

 \

\noindent
{\bf Case $\theta_3$}: 
%Condition {\bf (A)} fails. 
Let $g\in G$ be  such that $\pi^*(g)=\theta_3$. Up to conjugation by an element in $T(k)$, we may assume that $g$ acts on $\bP^3_{y_1,y_2,y_3,y_4}$ via 
$$
(y_1,y_2,y_3,y_4)\mapsto(y_4,y_1,y_2,y_3).
$$
One can check that for any 2-torsion element $\iota\in T(k)$,  the group $\langle g,\iota\rangle$ contains an abelian subgroup with no fixed point on $\bP^3$, contradiction. 

\ 

\noindent 
{\bf Case $\theta_4$}:  This element is contained in a $\fD_4$, covered in Lemma~\ref{lemm:d4}. 
\end{proof}

\begin{lemm} 
\label{lemm:c22}
Assume that $\pi^*(G)=C_2^2$. Then 
$$
{\bf (A)} \Longleftrightarrow {\bf (U)}, 
$$
unless $\pi^*(G)$ is conjugate to the group indicated in \eqref{eqn:C22}. 
\end{lemm}

\begin{proof} 
We have 9 conjugacy classes of $C_2^2$ in $\GL_3(\bZ)$ not containing $\eta$, denoted by 
$$
\fK_1,\ldots, \fK_9.
$$
We study their realizations:

 \ 

 \noindent
{\bf Cases  $\fK_1$ and $\fK_2$:} Here, $X=\bP^1\times \bP^1\times \bP^1$ and  
%Assuming the $G$-action does not switch the factors, there are 2 conjugacy classes of $\pi^*(G)$
    \begin{align*}
    \fK_1=&\langle\mathrm{diag}(-1,-1,1),\mathrm{diag}(-1,1,-1)\rangle,\\
   \fK_2=&\langle\mathrm{diag}(1,1,-1),\mathrm{diag}(-1,1,-1)\rangle.
     \end{align*}
    Using Remark~\ref{rema:alt}, we see that Condition {\bf (A)} fails if $\pi^*(G)=\mathfrak K_1$. 
    When $\pi^*(G)=\mathfrak K_2$, Condition {\bf (A)} implies that
    $$
    G_T\subset\{(1,t,1):t\in\bG_m(k)\}\subset T(k).
    $$
    In this case, $X^G\ne\emptyset$ and ${\bf (U)}$ holds.

\ 

\noindent 
{\bf Cases   $\fK_3$, $\fK_4$, and $\fK_5$}: Here, $X=\bP^1\times Q$, with $G$ switching the factors in  $Q=\bP^1\times \bP^1$. The groups are  
 {\tiny  $$
       \mathfrak K_3=\left\langle\begin{pmatrix}
            1&0&0\\
            0&-1&0\\
            0&0&1
        \end{pmatrix},\begin{pmatrix}
            0&0&1\\
            0&1&0\\
            1&0&0     
\end{pmatrix}\right\rangle  ,\quad \mathfrak K_4=\left\langle\begin{pmatrix}
            -1&0&0\\
            0&1&0\\
            0&0&-1
        \end{pmatrix},\begin{pmatrix}
            0&0&1\\
            0&1&0\\
            1&0&0     
    \end{pmatrix}\right\rangle,
    $$}
{\tiny    $$ 
    \mathfrak K_5=\left\langle\begin{pmatrix}
            -1&0&0\\
            0&1&0\\
            0&0&-1
        \end{pmatrix},\begin{pmatrix}
            0&0&1\\
            0&-1&0\\
            1&0&0     
    \end{pmatrix}\right\rangle.
    $$}

 \ 
 
 \noindent  
When $\pi^*(G)=\mathfrak K_3$,  Condition {\bf (A)} implies that
    $$
    G_T\subset\{(t_1,1,t_2):t_1,t_2\in\bG_m(k)\}\subset T(k).
    $$
When $\pi^*(G)=\mathfrak K_4$, 
    $$
    G_T\subset\{(1,t,1):t\in\bG_m(k)\}\subset T(k).
    $$
In both cases, $X^G\neq\emptyset$ and thus {\bf (U)}. 
However,  when $\pi^*(G)=\mathfrak K_5$, the subgroup generated by $$
    (1,-1,1)\in G_T,\quad \text{and a lift to $G$ of }{\tiny \begin{pmatrix}
        0&0&1\\
        0&-1&0\\
        1&0&0
    \end{pmatrix}}$$ 
    is an abelian group without fixed points. 

\ 

    \noindent 
{\bf Cases $\fK_6$ and $\fK_7$}: Here, $X=\bP^3$, and  
  {\tiny  $$ 
   \mathfrak K_6= \left\langle\begin{pmatrix} 0&0&1\\
    -1&-1&-1\\
     1& 0& 0
     \end{pmatrix},
\begin{pmatrix}
    -1&-1&-1\\
     0& 0& 1\\
     0&1&0
     \end{pmatrix}\right\rangle,\quad
      \mathfrak K_7=   \left\langle \begin{pmatrix}
       0&  1 &0\\
    1 &0 &0\\
 0 & 0 & 1
\end{pmatrix},
   \begin{pmatrix}  1 &0 &0\\
    0 &1 &0\\
     -1 & -1 & -1
   \end{pmatrix}\right\rangle.
   $$}

\

 When $\pi^*(G)=\mathfrak K_6$, then, up to conjugation, the $G$-action on $\bP^3_{y_1,y_2,y_3,y_4}$ is given by $G_T$ and a lift of $\fK_6$ generated by
   \begin{align*}
    g_1: (\mathbf y)&\mapsto (y_3,ay_4,y_1,ay_2),\quad a\in \bG_m(k),\\
   g_2: (\mathbf y)&\mapsto  (y_4,b_1y_3,b_1y_2,b_2y_1),\quad b_1,b_2\in \bG_m(k).
   \end{align*}
For $I\subset\{1,2,3,4\}$, let $s_{I}$ be the diagonal matrix changing the signs of $y_i,i\in I$. The abelian groups
$$
\langle g_1, s_{\{1,2\}}\rangle,\quad \langle g_2,s_{\{1,3\}}\rangle,\quad
\langle g_1g_2, s_{\{1,4\}}\rangle
$$
have no fixed points on $\bP^3$. On the other hand, we have that
$$
(g_1g_2)^2=\mathrm{diag}(1,1,b_2,b_2).
$$
If $b_2\ne1$, then $s_{\{1,2\}}\in G_T$. If $b_2=1$, then
$$
(s_{\{i\}}g_1g_2)^2=s_{\{1,2\}}\in G_T,
$$
for any $i=1,2,3$, or $4$ such that $s_{\{i\}}\in G_T$. Thus, in all cases, Condition {\bf (A)} fails.

\

When $\pi^*(G)=\mathfrak K_7$, $G$ is generated by $G_T$, and  
\begin{align*}
    g_3: (\mathbf y)\mapsto (c_1y_1,c_2y_2,y_4,y_3),\quad c_1,c_2\in\bG_m(k),\\
     g_4: (\mathbf y)\mapsto (y_2,y_1,c_3y_3,c_4y_4),\quad c_3,c_4\in\bG_m(k). 
\end{align*}
Note that the abelian group generated by
$$
\langle g_3 g_4,\mathrm{diag}(1,-1,a_1,-a_1)\rangle,
$$
does not fix any points on $\bP^3$, for any $a_1\in\bG_m(k)$.

If $c_1\ne\pm c_2$ and $c_3\ne \pm c_4$, then $g_3$ and $g_4$ generate one of 
$$
\mathrm{diag}(1,-1,1,-1) \quad \text{and}\quad\mathrm{diag}(1,-1,-1,1),
$$
and Condition {\bf (A)} fails. Thus, we may assume that $c_1=\pm c_2$ and  all elements in $G_T$ are of the form 
$$
\mathrm{diag}(1,\pm1,a_3,a_4),\quad a_3,a_4\in\bG_m(k), \quad a_3\ne -a_4.
$$
If all elements in $G_T$ are of the form 
$
\mathrm{diag}(1,1,a_3,a_4),
$ 
then $\langle G_T,g_4\rangle$ is an abelian subgroup of $G$ of index 2. It follows that the Bogomolov multiplier $\rB^2(G,k^\times)=0$ and Condition {\bf (A)} implies {\bf (U)}.

Now, we consider the case when $G_T$ contains elements of the form 
$$
\mathrm{diag}(1,-1,a_3,a_4).
$$
Up to multiplying $g_3$ with such an element, we may assume that $c_1=-c_2$. 
 We divide the argument into the following subcases:
\begin{enumerate}
\item 
When
$c_1=-c_2$, $c_3=c_4$ and all elements in $G_T$ are of the form 
$$
\mathrm{diag}(1,\pm1,a_3,a_3),
$$
   then $G$ fixes $[0:0:1:1]\in\bP^3$.
\item 
When
$c_1=-c_2$ and $c_3=-c_4$, then $\langle g_3,g_4\rangle$ is an abelian group with no fixed points on $\bP^3.$
\item 
When $c_1=-c_2$, $c_3\ne \pm c_4$ and $G_T$ contains an element 
$$
\varepsilon=\mathrm{diag}(1,\pm1,a_3,a_4),\quad a_3,a_4\in\bG_m
$$
where $\mathrm{ord}(\frac{a_3}{a_4})\geq\mathrm{ord}(-\frac{c_4}{c_3})=\mathrm{ord}(\frac{c_4}{c_3})$, then there exists $n\in\bZ$ such that 
$$
\frac{a_3^n}{a_4^n}=-\frac{c_4}{c_3},\quad \text{i.e.,}\quad a_3^nc_3=-a_4^nc_4.
$$
We are reduced to the previous case. In particular, the abelian group 
$
\langle g_3,\varepsilon^ng_4\rangle
$
has no fixed points on $\bP^3$.

\item When $c_1=-c_2$, $c_3\ne \pm c_4$ and $G_T$ contains an element 
$$
\varepsilon=\mathrm{diag}(1,-1,a_3,a_4),\quad a_3,a_4\in\bG_m(k), \quad a_3\ne \pm a_4
$$
where $\mathrm{ord}(\frac{c_4}{c_3})>\mathrm{ord}(\frac{a_3}{a_4})=-\mathrm{ord}(\frac{a_3}{a_4})$, then there exists $n\in\bZ$ such that 
$$
\frac{c_4^{2n}}{c_3^{2n}}=-\frac{a_3}{a_4},
$$
and thus 
$$
\varepsilon\cdot g_4^{2n}=\mathrm{diag}(1,-1,a_3c_3^{2n},-a_3c_3^{2n}).
$$
By the observation above, the abelian group $\langle g_3g_4,\varepsilon g_4^{2n}\rangle$ does not fix points on $\bP^3$. 
\end{enumerate}
Thus, Condition {\bf (A)} implies {\bf (U)} when $\pi^*(G)=\fK_7$.
%{\bf check this is ok when $G_T$ has no extra element!}

\

\noindent
{\bf Case $\fK_8$}: The group is given by 
    {\tiny$$
     \mathfrak K_8=   \left\langle \begin{pmatrix}
       1&  1 &1\\
    0 &0 &-1\\
 0 & -1 & 0
\end{pmatrix},
   \begin{pmatrix}  0 &0 &1\\
    -1 &-1 &-1\\
     1 & 0 & 0
   \end{pmatrix}\right\rangle.
   $$}

  \noindent This is a subgroup contained in a $\fD_4$, covered in Lemma~\ref{lemm:d4}.

\

\noindent
 {\bf Case $\fK_9$}: This is the exceptional case. The group is given by {\tiny$$
     \mathfrak K_9=   \left\langle    \begin{pmatrix}  0 &1 &-1\\
    1 &0 &-1\\
     0 & 0 & -1
   \end{pmatrix},\begin{pmatrix}
       -1&  0 &0\\
    -1 &0 &1\\
 -1 & 1 & 0
\end{pmatrix}
\right\rangle.
   $$} 
   
   \noindent In Section~\ref{sect:exept}, we show that $\beta(X,G)\ne 0$ for all $G$ with $\pi^*(G)=\mathfrak K_9$.

\end{proof}

\begin{lemm} 
\label{lemm:d4}
Assume that $\pi^*(G)\subseteq \fD_4=\langle \tau_1,\tau_2\rangle$, 
where
$$
{\tiny\tau_1=\begin{pmatrix}
            1&0&0\\
            -1&-1&-1\\
            0&0&1
        \end{pmatrix},
        \quad
        \tau_2=\begin{pmatrix}
            1&1&1\\
            -1&0&0\\
            0&-1&0     
    \end{pmatrix}.}
    $$ 
Then 
$$
{\bf (A)} \Longleftrightarrow {\bf (U)}.
$$
\end{lemm}
\begin{proof}
We first assume that 
$\pi^*(G)=\langle\tau_1,\tau_2\rangle$. 
As explained in Section~\ref{sect:3-folds}, a simpler smooth projective model $X$ in this case is the blowup of a quadric cone at its vertex. In particular, we have 
$$
\Pic(X)=\bZ\oplus \rP,\quad \rP=\bZ\oplus \bZ
$$ 
with $G$ acting trivially on the first summand, and switching two factors of  the second summand $\rP$.  

 Let $\Sigma$ be the fan of $X$ given in Section~\ref{sect:3-folds}.  We note that $\pi^*(G)$ acts trivially on the 1-dimensional sublattice $\rN'\subset \rN$ spanned by the ray $v_1=(-1,0,-1)$. Let $\sigma\in \Sigma$ be the cone generated by $v_5$. It corresponds to a $G$-invariant toric boundary divisor $D_\sigma\subset X$. 
On the other hand, the sublattice $N'$ also gives rise to a quotient torus, cf. \cite[Section 2.3]{KT-toric}. In particular, we have
$$
(\rN/\rN')^\vee=\sigma^\perp\cap \rM\simeq \bZ^2.
$$
For any cone $\sigma'\in \Sigma$ such that $\sigma'\supseteq\sigma$, put 
$$
\bar\sigma':=(\sigma'+\bR\sigma)/\bR\sigma\subset (\rN/\rN')_\bR.
$$
All such $\bar\sigma'$ form a new $G$-invariant fan $\Sigma_{\sigma}$. Let $X(\Sigma_\sigma)$ be the toric variety associated with $\Sigma_{\sigma}$. One can check that $X(\Sigma_\sigma)=\bP^1\times\bP^1$. 

By \cite[Section 2.3]{KT-toric}, $X(\Sigma_\sigma)$ is $G$-isomorphic to $D_\sigma$, and there exists a $G$-equivariant rational map
$$
\rho: X\dashrightarrow X({\Sigma}_\sigma)\simeq D_\sigma.
$$ 
Note that the $G$-action on $D_\sigma$ is not necessarily generically free. The map $\rho$ induces a homomorphism of $G$-lattices
$$
\rho^*: \Pic(D_\sigma)\to\Pic(X).
$$ This yields a commutative diagram of $G$-modules
$$
{\xymatrix{
0\ar[r] &\rM\ar[r]&\mathrm{PL}(X)\ar[r]&\Pic(X)\ar[r]&0\\
0\ar[r] &(\rN/\rN')^\vee\ar[r]\ar[u]&\mathrm{PL}(D_\sigma)\ar[r]\ar[u]&\Pic(D_\sigma)\ar[r]\ar[u]^{\rho^*}&0\\
}
}
$$
where $\mathrm{PL}(X)=\bZ^6$ and $\mathrm{PL}(D_\sigma)=\bZ^4$. Following the diagram, one sees that the dual map 
$$
(\rho^*)^\vee:\Pic(X)^\vee\to\Pic(D_\sigma)^\vee
$$
can be identified with the canonical projection (note that $\Pic(X)$ is self-dual under the $G$-action) 
$$
\rP\oplus \bZ\to \rP.
$$

Now assume that the $G$-action on $X$ satisfies Condition {\bf (A)}. Let
 $$
 \beta:=\beta(X,G)\in\rH^2(G,\Pic(X)^\vee\otimes k^\times),
 $$ 
 and $H$ be the maximal subgroup of $G$ such that $\pi^*(H)=\langle\tau_1,\tau_2^2\rangle\simeq C_2^2$. We have that $[G:H]=2$ and $H$ acts trivially on $\rP$; in particular,
$$
\rP=\mathrm{Ind}_H^G(\bZ)
$$ 
is the $G$-module induced from the trivial $H$-module $\bZ$.   Consider the commutative diagram 
 $$
 {\xymatrix{
 \rH^2(G,(\rP\oplus\bZ)\otimes k^\times)\ar[d]^{\mathrm{res}_1}\ar[rr]^{\quad\mathrm{pr}_1}&&\rH^2(G,\rP\otimes k^\times)\ar[d]^{\mathrm{res}_2}\\
 \rH^2(H,(\rP\oplus\bZ)\otimes k^\times)\ar[rr]^{\quad\mathrm{pr}_2}&& \rH^2(H,\rP\otimes k^\times)
 }}
$$
where $\res_1$ and $\res_2$ are the corresponding restriction homomorphisms, and $\mathrm{pr}_1$ and $\mathrm{pr}_2$ are projections induced by $(\rho^*)^\vee$. By functoriality,  
$$
\mathrm{pr}_1(\beta)=\beta(D_\sigma,G)
$$
where $\beta(D_\sigma,G)$ is the class corresponding to the $G$-action on $D_\sigma$.

Since $\pi^*(H)=C_2^2$ is conjugate to $\mathfrak K_7$, by the proof of Lemma~\ref{lemm:c22}, Condition {\bf (A)} implies that the $H$-action on $X$ is {\bf (U)}, and thus
$$
\res_1(\beta)=0, \quad \mathrm{pr}_2(\res_1(\beta))=0.
$$ 
By Lemma~\ref{lemm:bog}, we know that $\res_2$ is injective. It follows that 
$$
\mathrm{pr}_1(\beta)=0
$$
and thus the $G$-action on $D_{\sigma}$ is {\bf(U)}. Now let 
$$
\varrho: X\to \bar X
$$
be the contraction of the boundary divisor in $X$ corresponding to the ray $v_6=(1,0,1)$. Then $\bar X\subset\bP^4$ is a cone over a smooth quadric surface, and $\varrho$ is the blowup of its vertex. The strict transform $\rho_*(D_\sigma)$ is a $G$-equivariantly unirational surface in $\bar X$.

Finally, the same argument as in \cite[Proposition 3.1]{CTZ-uni} shows that the $G$-action on $\bar X$ is {\bf (U)}: we have a $G$-equivariant dominant rational map 
$$
\varrho_*(D_{\sigma})\times\bP^4\dashrightarrow \bar X,
$$
sending the pair of points $(q_1,q_2)\in \varrho_*(D_{\sigma})\times\bP^4$ to the second intersection point of $X$ with the line passing through $q_1$ and $q_2$. It follows that the $G$-action on $X$ is also {\bf (U)}.

The same proof applies when $\pi^*(G)$ is a subgroup of $\fD_4=\langle\tau_1,\tau_2\rangle$ and $G$ swaps the two factors of $\rP$. When $G$ does not swap the two factors, $\pi^*(G)$ has been already covered by previous lemmas.
\end{proof}

\begin{lemm} 
\label{lemm:d4sum}
Assume that $\pi^*(G)=\fD_4$. Then 
$$
{\bf (A)} \Longleftrightarrow {\bf (U)}.
$$
\end{lemm}
\begin{proof}  
There are $8$ conjugacy classes of $\fD_4$ in $\GL_3(\bZ)$: up to conjugation, two of them contain $\theta_2$; two of them contain $\theta_3$; among the rest, one contains $\mathfrak K_1$ and one contains $\mathfrak K_6$. From the analysis above, we know that Condition {\bf (A)} fails for these 6 classes.

 One of the two remaining classes is covered by Lemma~\ref{lemm:d4}. In the other case, $\pi^*(G)$ is generated by 
 $$
 \iota_2={\tiny\begin{pmatrix}
    1&0&0\\
    0&1&0\\
    0&0&-1
\end{pmatrix}}\quad \text{and}\quad\theta_1={\tiny\begin{pmatrix}
     1&0&0\\
     0&0&1\\
     0&-1&0
 \end{pmatrix}.
} $$
 This is realized on $X=(\bP^1)^3$, where
 $$
 \Pic(X)=\bZ\oplus \rP,\quad \rP=\bZ\oplus\bZ,
 $$
 $\iota_2$ acts trivially on $\Pic(X)$ and $\theta_1$ switches the two factors of $\rP$.
 Since $\pi^*(G)$ contains a subgroup conjugated in $\GL_3(\bZ)$ to $\mathfrak K_4$, we know that 
 $$
 G_T\subset\{(t,1,1):t\in\bG_m(k)\}\subset T(k).
 $$
 Let $H$ be the subgroup of $G$ generated by $G_T$ and lifts to $G$ of $\iota_2$ and $\theta_1^2$. It follows that $H$ is abelian and $[G:H]=2$. Thus, 
 $$\rB^2(G,k^\times)=0,$$ and  Condition {\bf (A)} implies that 
$$
\beta(X,G)\in\rB^2(G,k^\times\otimes\Pic(X))\simeq\rB^2(G,k^\times\otimes\rP).
 $$
 Note that $H$ acts trivially on $\Pic(X)$ and $\rP=\mathrm{Ind}_H^G(\bZ)$ for the trivial $G$-module $\bZ$. Since $H$ is abelian,  we have that $\rB^2(H,k^\times\otimes\rP)=0$.  Lemma~\ref{lemm:bog} shows that $\rB^2(G,k^\times\otimes\rP)=0$. Thus, we conclude that $\beta(X,G)=0$.
\end{proof}

\section{The exceptional case $\fK_9$}
\label{sect:exept}
This section is devoted to a proof of the following lemma, which completes the proof of Lemma~\ref{lemm:c22}.

\begin{lemm} \label{lemm:k9}
Assume that $G_T\neq1$ and $\pi^*(G)$ contains a subgroup conjugated in $\GL_3(\bZ)$ to 
$$
     \mathfrak K_9=  {\tiny \left\langle    \begin{pmatrix}  0 &1 &-1\\
    1 &0 &-1\\
     0 & 0 & -1
   \end{pmatrix},\begin{pmatrix}
       -1&  0 &0\\
    -1 &0 &1\\
 -1 & 1 & 0
\end{pmatrix}
\right\rangle.}
   $$
Then the $G$-action on a smooth projective model $X$ fails {\bf (U)}. 
\end{lemm}

\begin{proof}
We may assume that $\pi^*(G)=\mathfrak K_9\simeq C_2^2$, and no proper subgroup of $G$ surjects to $C_2^2$ via $\pi^*$. Then $G$ is generated by 
$$
\sigma_1: (t_1,t_2,t_3)\mapsto (b_1t_2,b_2t_1,\frac{b_3}{t_1t_2t_3}),\quad \sigma_2:(t_1,t_2,t_3)\mapsto (\frac{c_1}{t_1t_2t_3},c_2t_3,c_3t_2).
$$
 The torus part $G_T$ is generated by 
$$
\sigma_1^2=\mathrm{diag}(b_1b_2,b_1b_2,\frac1{b_1b_2}),\quad 
\sigma_2^2=\mathrm{diag}(\frac1{c_2c_3},c_2c_3,c_2c_3),
$$
and 
$$
(\sigma_1\sigma_2)^2=\mathrm{diag}(\frac{c_1c_3}{b_1b_3},\frac{b_1b_3}{c_1c_3},\frac{c_1c_3}{b_1b_3}). 
$$
Since $G$ is finite, we know that $b_1b_2, c_2c_3$ and ${c_1c_3}/{b_1b_3}$ have finite orders. Up to a change of variables 
$$
t_1\mapsto r_1t_1,\quad t_2\mapsto r_2t_2,\quad t_3\mapsto r_3t_3,
$$
where $r_1,r_2,r_3\in k^\times$ are such that
$$
r_1b_1=r_2,\quad r_1r_2r_3^2b_3=r_1^2r_2r_3c_1=1,
$$
we may assume that $b_1=b_3=c_1=1$, and $b_2,c_2,c_3$ are roots of unity whose orders are powers of 2.

When the $G$-action satisfies Condition {\bf (A)}, we know that $G_T$ is cyclic. Indeed, if $G_T$ is not cyclic, then $G$ contains one of the following subgroups which fail Condition {\bf(A)}:
$$
\langle\sigma_1,\mathrm{diag}(1,1,-1)\rangle, \quad \langle\sigma_1,\mathrm{diag}(-1,-1,1)\rangle,
$$
$$
\langle\sigma_2,\mathrm{diag}(-1,1,1)\rangle\quad \langle\sigma_2,\mathrm{diag}(1,-1,-1)\rangle,
$$
$$
\langle\sigma_1\sigma_2,\mathrm{diag}(1,-1,1)\rangle\quad \langle\sigma_1\sigma_2,\mathrm{diag}(-1,1,-1)\rangle.
$$
From this, we see that at least two of 
$$
\sigma_1^2,\quad \sigma_2^2,\quad (\sigma_1\sigma_2)^2
$$ 
have order 1 or 2. Up to a permutation of coordinates, we may assume that the latter two have order 1 or 2. One can check that the $G$-action then satisfies Condition {\bf (A)}. Let $n\in\bZ$ such that $b_2$ has order $2^{n-2}$. There are three cases:
\begin{enumerate}
    \item  $c_2=1,c_3=-1$ : in this case $G\simeq \fQ_{2^n}$,
    \item  $c_2=c_3=1$ : in this case $G\simeq \fD_{2^{n-1}}$,
    \item  $c_2=c_3=-1$ : in this case $G\simeq \mathfrak{SD}_{2^n}$.
\end{enumerate}

 In each case, we have $\sigma_1=x$ and $\sigma_2=y$, where $x,y$ are the same as in the presentations of $G$, with generators and relations, given in Section~\ref{sect:gene}. Using the   resolutions~\eqref{eqn:resolutionQ8},~\eqref{eqn:resolutiondihedral}, and~\eqref{eqn:resolutionsemidihedral}, we compute $\beta(X,G)$, 
 as an element in $\rH^3(G,\Pic(X)^\vee)$, following the recipe in
Section~\ref{sect:gene}.

We recall the presentation of  $\Pic:=\Pic(X)$ on the smooth projective model {\bf (S)}, via the exact sequence
$$
0\to \rM\to \mathrm{PL} \to \Pic \to 0,
$$
where $\mathrm{PL}=\bZ^{14}$ is generated by the following 14 rays:
$$
   ( -1, 0, 0 ),
    ( -1, 1, 0 ),
    ( 0, -1, 1 ),
    ( 0, 0, -1 ),
    ( 0, 0, 1 ),
    ( 0, 1, -1 ),( 1, -1, 0 ),
$$
$$
    ( 1, 0, 0 ),
    ( 1, 0, -1 ),
    ( 1, -1, 1 ),
    ( 0, -1, 0 ),
    ( 0, 1, 0 ),
    ( -1, 1, -1 ),
    ( -1, 0, 1 ),
$$
labeled by $v_i,i =1,\ldots,14$, in order. The character lattice $\rM$ is embedded in $\mathrm{PL}$ as a submodule with basis
\begin{align*}
m_1&=-v_{1} - v_{2} + v_{7} + v_{8} + v_{9} + v_{10} - v_{13} - v_{14},\\
m_2&=v_{2} - v_{3} + v_{6} - v_{7} - v_{10} - v_{11} + v_{12} + v_{13},\\
m_3&=v_{3} - v_{4} + v_{5} - v_{6} - v_{9} + v_{10} - v_{13} + v_{14}.
\end{align*}
Let $p_{1}\ldots,p_{11}$ be a basis of $\Pic$, and  $e_{1},\ldots,e_{11}$ the corresponding dual basis of $\Pic^\vee$, such that $p_i$ can be lifted to $\mathrm{PL}$ by 
\begin{align}\label{eqn:liftpL}
    p_i\mapsto v_{4+i},\quad i=1,2,\ldots,9,
\end{align}
$$
 p_{10}\mapsto(v_4-v_5), \quad  p_{11}\mapsto(-v_4+v_5-v_9+v_{14}).
$$
Observe that the resolutions~\eqref{eqn:resolutionQ8},~\eqref{eqn:resolutiondihedral}, and~\eqref{eqn:resolutionsemidihedral} start with the same first step. Indeed, the images of $\mathrm{id}_\Pic$ in $\rH^1(G,\rM\otimes\Pic^\vee)$ are the same in all three cases of $G$. We first compute this intermediate class via
  \begin{align}\label{eqn:firstseq}
 {\small\xymatrix@C=1em{
\mathrm{id}_{\Pic}\in\Pic\otimes\Pic^\vee&\mathrm{PL}\otimes\Pic^\vee\ar[l]\ar[dd]^{\tiny{\begin{pmatrix}
     1-x&1-y
 \end{pmatrix}}}&&\\
 &&&&\\ 
 &(\PL\otimes\Pic^\vee)^2&(\rM\otimes\Pic^\vee)^2\ar@{_{(}->}[l]&(k(T)^\times\otimes\Pic^\vee)^2\ar[l]
 }}
 \end{align}
We choose a lift of $\mathrm{id}_{\Pic}$ to $\PL\otimes\Pic^\vee$ given by \eqref{eqn:liftpL}. The resulting class in $(\rM\otimes\Pic^\vee)^2$ is 
\begin{multline}\label{eqn:intcls}
   ( (-m_2-m_3)\otimes e_3, \quad (m_2+m_3)\otimes e_1+m_2\otimes e_2+m_1\otimes e_4+\\+(-m_2-m_3)\otimes e_{10}+(m_2+m_3)\otimes e_{11}).
\end{multline}
This class represents the image of $\mathrm{id}_{\Pic}$ in $\rH^1(G,\rM\otimes\Pic^\vee)$. We lift this class to $(\rM\otimes\Pic^\vee)^2$ via the set-theoretic map $\rM\to k(T)^\times$ given by 
$$
(a_1,a_2,a_3)\to t_1^{a_1}t_2^{a_2}t_3^{a_3}.
$$
The next step of the computation depends on the isomorphism class of $G$. We proceed case-by-case. 

 For $G=\fQ_{2^n}$, we continue \eqref{eqn:firstseq} via 
 $$
{\small \xymatrix{
(k(T)^\times\otimes\Pic^\vee)^2\ar[rrr]^{\tiny\begin{pmatrix}
  N_x& \!\! {yx}+1\\
  -1-y&\!\!x-1
  \end{pmatrix}}&&&(k(T)^\times\otimes\Pic^\vee)^2&&\\
  &&&(k^\times\otimes\Pic^\vee)^2\ar@{^{(}->}[u]
  &&\\
  &&&(\bQ\otimes\Pic^\vee)^2\ar[u]\ar[rr]^{\tiny \begin{pmatrix}
 1-x \\yx-1 \end{pmatrix}}&&\bQ\otimes\Pic^\vee&\\
  &&&&&\bZ\otimes\Pic^\vee\ar@{^{(}->}[u].
  }}
  $$
The first two arrows come from the resolution \eqref{eqn:resolutionQ8}. The rest of the diagram arises from the exact sequence 
  $$
  0\to\bZ\to\bQ\to\bQ/\bZ\to 0.
  $$
Since $G$ is a finite group, we may replace $k^\times$ by $\bQ/\bZ$, via the map 
    $$
 \{x\in k^\times: \text{the order of $x$ is finite}\}\to \bQ/\bZ,\quad x\mapsto \frac{\log(x)}{2\pi i}.
  $$
  The class $\beta(X,G)\in\rH^2(G,k^\times\otimes\Pic^\vee)$ is represented by the image of \eqref{eqn:intcls} in $(k^\times\otimes\Pic^\vee)^2$ via the diagram above. This image is 
\begin{multline}\label{eqn:betaclassq8}
   ( -1\otimes(e_1+e_2+e_4+e_{10}+e_{11})+(b_2^{-2^{n-3}})\otimes e_3,\\
b_2\otimes(e_4+e_{10}-e_3-e_{11})+(-b_2^{-1})\otimes e_1).
\end{multline}
 To determine whether or not this class vanishes, we map it further to $\rH^3(G,\bZ\otimes\Pic^\vee)$ using the sequence above. In particular, we choose a lift $k^\times\to \bQ$ such that 
 $$
 -1\mapsto \frac12,\quad b_2\mapsto \frac{\log(b_2)}{2\pi i}, \quad -b_2^{-1}\mapsto\frac12-\frac{\log(b_2)}{2\pi i},
 $$
 $$
 b_2^{-2^{n-3}}\mapsto -\frac{\log(b_2)}{2\pi i}\cdot2^{{n-3}}, \quad \text{ where } \quad 
0\leq  \frac{\log(b_2)}{2\pi i}<1.
 $$
 Using this choice, the image of \eqref{eqn:betaclassq8} in $\bZ\otimes\Pic^\vee=\Pic^\vee$, following the diagram above, is given by 
$$
\beta=(-1,  0 , 1,  0 , 0 , 0 ,-1 , 0,  0,  1 , 0),
$$
under the basis $e_1,\ldots,e_{11}.$
On the other hand, the image of $$
\nu: (\Pic^\vee)^2\stackrel{\small\begin{pmatrix}
  1-x \\yx-1 \end{pmatrix}}{\xrightarrow{\hspace*{2cm}}}\Pic^\vee
  $$
  is the $\bZ[G]$-module generated by 
  $$
( 0,  0,  0,  0,  0,  0,  2, 0, 0, -2,  0),
$$
  $$( 1,  0,  0,  0,  0,  0,  1,  0,  0, -2,  1),\quad 
( 0,  1,  0,  0,  0,  0,  1, -1,  0, -1,  1),
$$
$$
( 0,  0,  1,  0,  0,  0,  1,  0,  0, -2,  1),\quad
( 0,  0,  0,  1,  0,  0,  1, -1,  0, -1,  1)
$$
$$
( 0,  0,  0,  0,  1,  0,  1, -1,  0, -2,  0),\quad
( 0,  0,  0,  0,  0,  1,  0,  0, -1,  0,  0).
$$
One can check that $\beta\notin\mathrm{im}(\nu)$. It follows that $\beta (X,G)\ne 0$.

\

Similarly, when $G=\fD_{2^{n-1}}$, using its resolution \eqref{eqn:resolutiondihedral}, we continue \eqref{eqn:firstseq} via
$$
{ \tiny\xymatrix{
(k(T)^\times\otimes\Pic^\vee)^2\ar[rrr]^{\tiny\begin{pmatrix}
  N_x& \!\! \!\! 1+{yx}&\!\! \!\!0\\
  0&\!\! \!\!x-1&\!\! \!\!1+y
  \end{pmatrix}
}&&&(k(T)^\times\otimes\Pic^\vee)^3&&\\
  &&&(k^\times\otimes\Pic^\vee)^3\ar@{^{(}->}[u]
  &&&\\
  &&&(\bQ\otimes\Pic^\vee)^3\ar[u]\ar[rrr]^{\tiny\begin{pmatrix}
  1-x& \!\! \!\! y+1&\!\! \!\!0&\!\! \!\!0\\ 
  0& \!\! \!\! -N_{x}&\!\! \!\!1-{yx}&\!\! \!\!0\\
   0& \!\! \!\! 0&\!\! \!\!1-x&\!\! \!\!1-y   \end{pmatrix}}&&&(\bQ\otimes\Pic^\vee)^4&\\
  &&&&&&(\Pic^\vee)^4\ar@{^{(}->}[u].
  }}
  $$
  The class $\beta(X,G)$ is represented in $(k^\times\otimes\Pic^\vee)^3$ by 
$$
      \left(b_2^{-2^{n-2}}\otimes e_3,\qquad b_2\otimes (e_4+e_{10}-e_1-e_3-e_{11}),\quad e\right),
      $$
      where $e$ is the identity element of $k^\times\otimes\Pic^\vee$.
 We choose a lift $k^\times\to\bQ$ such that 
$$
b_2\mapsto \frac{\log(b_2)}{2\pi i},\quad b_2^{-2^{n-2}}\mapsto -\frac{\log(b_2)}{2\pi i}\cdot2^{n-2},
$$
where 
 $$
0\leq \frac{\log(b_2)}{2\pi i}< 1.
 $$
With this choice, the image of $\beta(X,G)$ in $\rH^3(G,\Pic^\vee)$ is represented in $(\Pic^\vee)^4$ by 
$$
    \beta=( \mathbf 0,
    (0, -\frac{\log(b_2)}{2\pi i}\cdot2^{n-2}, 0, -\frac{\log(b_2)}{2\pi i}\cdot 2^{n-2}, 0, 0, \\
    0, 0, 0, 0, 0),\mathbf 0,\mathbf 0,\mathbf 0),
$$
where $\mathbf 0$ denotes the zero element in $\Pic^\vee$.
Note that 
$$
-\frac{\log(b_2)}{2\pi i}\cdot 2^{n-2}
$$ is an odd integer since $\mathrm{ord}(b_2)=2^{n-2}$.
On the other hand, the intersection of the image of 
$$
\nu: (\Pic^\vee)^3\stackrel{\tiny\begin{pmatrix}
  1-x& \!\! \!\! y+1&\!\! \!\!0&\!\! \!\!0\\ 
  0& \!\! \!\! -N_{x}&\!\! \!\!1-{yx}&\!\! \!\!0\\
   0& \!\! \!\! 0&\!\! \!\!1-x&\!\! \!\!1-y   \end{pmatrix}}{\xrightarrow{\hspace*{4cm}}}(\Pic^\vee)^4
$$
with the subspace $\langle\mathbf 0\rangle\times \langle\mathbf 0\rangle\times \Pic^\vee\times\langle\mathbf 0\rangle$
is generated by the following elements in $\Pic^\vee$:
$$
   ( 1, 0, 1, 0, 0, 0, 0, 0, 0, 0, 0 ),\quad
    ( 0, 1, 0, 1, 1, 0, 1, 1, 0, 0, 0 ),
$$
$$(
0, 0, 0, 0, 2, 0, 2, 2, 0, 0, 0),\quad
    ( 0, 0, 0, 0, 0, 2, 0, 0, 2, 0, 0 ).
$$
One can check that $\beta\not\in\mathrm{im}(\nu)$ and thus $\beta(X,G)\ne 0$.

\

Finally, for $G=\mathfrak{SD}_{2^n}$, using its resolution \eqref{eqn:resolutionsemidihedral}, we continue \eqref{eqn:firstseq} via
$$
{\small \xymatrix{
(k(T)^\times\otimes\Pic^\vee)^2\ar[rr]^{\tiny\begin{pmatrix}
  L_2& 0\\
  L_1& y+1
  \end{pmatrix}
}&&(k(T)^\times\otimes\Pic^\vee)^2&&\\
  &&(k^\times\otimes\Pic^\vee)^2\ar@{^{(}->}[u]
  &&\\
  &&(\bQ\otimes\Pic^\vee)^2\ar[u]\ar[rr]^{\tiny \begin{pmatrix}
  -L_3& \!\!0\\
  L_4&\!\!1-y
  \end{pmatrix}}&&(\bQ\otimes\Pic^\vee)^2&\\
  &&&&(\Pic^\vee)^2\ar@{^{(}->}[u]
  }}
  $$
The class $\beta(X,G)$ is represented in $(k^\times\otimes\Pic^\vee)^2$ by 
\begin{multline*}
    (b_2^{2^{n-4}-1}\otimes e_1+{b_2^{2^{n-4}}}\otimes e_2+b_2^{2^{n-4}+1}\otimes (e_4-e_3)+b_2(e_{10}-e_{11}),\quad -1\otimes e_2).
\end{multline*}
We choose a lift of $k^\times\to \bQ$ such that 
\begin{align*}
-1\mapsto \frac12,\quad b_2^r\mapsto \frac{\log(b_2)}{2\pi i}\cdot r,\quad\forall r\in\bZ,
\end{align*}
where 
 $$
0\leq{\log(b_2)}<{2\pi i}.
 $$
Under this lift, the image of $\beta(X,G)$ in $\rH^3(G,\Pic^\vee)$ is represented in $(\Pic^\vee)^2$ by 
$$
\beta=( (0,1,0,-1,0,0,0,0,0,0),
    (0, 0, 0, 0, 0, 0, 0, 0, 0, 0, 0)).
$$
We find that $\beta$ is not in the image of 
$$
\nu: (\Pic^\vee)^2\stackrel{\small\begin{pmatrix}
  -L_3& \!\!0\\
  L_4&\!\!1-y
  \end{pmatrix}}{\xrightarrow{\hspace*{2cm}}}(\Pic^\vee)^2.
$$
Indeed, one can check that $\beta$ is not in the intersection $\mathrm{im}(\nu)\cap (\Pic^\vee\times \mathbf{0})$, which is the $\bZ[G]$-module generated by 
$$
(( 1,  1, -1, -1,  1,  0,  1, -1,  0, -2,  0), (0,  0,  0 , 0,  0 , 0,  0,  0,  0 , 0 , 0)),
$$
$$
(( 0,  2,  0, -2,  0,  0,  0,  0,  0,  0,  0), (0,  0,  0,  0,  0,  0,  0,  0,  0,  0,  0)).
$$
We conclude that $\beta(X,G)\ne 0$. 
This completes the proof of Lemma~\ref{lemm:k9}.
\end{proof} 

\begin{exam}
\label{exam:bad}
Let $G=\fD_4$ be generated by 
$$
(t_1,t_2,t_3)\mapsto(t_2,-t_1,\frac{1}{t_1t_2t_3}),
\quad
(t_1,t_2,t_3)\mapsto(\frac1{t_1t_2t_3},t_3,t_2).
$$
Here,  $G_T=\langle(-1,-1,-1)\rangle\simeq C_2$ and  $\pi^*(G)=\fK_9$. 
The $G$-action satisfies {\bf (A)} -- the two noncyclic  $\fK_4\subset \fD_4$ map to $C_2$ via $\pi$, and  
fix points on the smooth model $(\bP^1)^3$.
However, $\beta(X,G)$ does not vanish, and the $G$-action fails {\bf (U)}. 
\end{exam}

\begin{rema}
    Using the analysis in Section~\ref{sect:uni}, one can check that when $\pi^*(G)$ strictly contains $\fK_9$, the $G$-action fails Condition {\bf (A)}. 
\end{rema}

\section{Stable linearizability}
\label{sect:final}

In this section, we prove Theorem~\ref{thm:main}, i.e., a criterion for unirationality and stable linearizability of generically free $G$-actions on toric threefolds. 
We assume the necessary Condition {\bf (A)}. 

 \

\noindent  
{\em Step 1.} 
By Theorem~\ref{thm:toricge}, for smooth projective toric varieties $X$, unirationality of the $G$-action is equivalent to the vanishing of the class 
$$
\beta(X, G)\in \rH^2(G,\Pic(X)^\vee\otimes k^\times)=\rH^3(G,\Pic(X)^\vee).  
$$
By Lemma~\ref{lemm:beta}, this class vanishes if and only if it vanishes upon restriction to every $p$-Sylow subgroup of $G$. 

\

\noindent  
{\em Step 2.} When $p\neq 2,3$, the $p$-Sylow subgroup of $G\subset \Aut(T)$ is a subgroup of translations $T(k)\subset\Aut(T)$, see \eqref{eqn:exte}. Since it has fixed points in the boundary $X\setminus T$, the action is unirational, by \eqref{eqn:xg}.

\

\noindent  
{\em Step 3.} For $p=3$, Proposition~\ref{prop:dim33group} implies that unirationality is equivalent to Condition {\bf (A)}.

\

\noindent  
{\em Step 4.} For $p=2$, Proposition~\ref{prop:dim32group} characterizes unirationality, as stated in Theorem~\ref{thm:main}.  Note that
a $G$-action with $\pi^*(G)$ as in \eqref{eqn:bad} fails Condition {\bf (SP)}.

\

\noindent  
{\em Step 5.} 
Stable linearizability of the action is governed by Theorem~\ref{thm:torict}:
assuming unirationality, {\bf (SL)} follows from 
the stable permutation property of $\Pic(X)$, as a $G$-module. 
This property only depends on the image $\pi^*(G)$ in $\GL(\rN)$. The
corresponding actions have been analyzed in \cite{kunyavskii}: 
the $G$-action on $\Pic(X)$ fails {\bf (SP)} if and only if $\pi^*(G)$ contains 
one of the three 2-groups indicated in Theorem~\ref{thm:main}.

\bibliographystyle{plain}
\bibliography{toric3d}

\begin{thebibliography}{10}

\bibitem{ademmilgram}
A.~Adem and R.~J. Milgram.
\newblock {\em Cohomology of finite groups}, volume 309 of {\em Grundlehren der
  mathematischen Wissenschaften [Fundamental Principles of Mathematical
  Sciences]}.
\newblock Springer-Verlag, Berlin, second edition, 2004.

\bibitem{BT-aniso}
V.~V. Batyrev and Yu. Tschinkel.
\newblock Rational points of bounded height on compactifications of anisotropic
  tori.
\newblock {\em Internat. Math. Res. Notices}, (12):591--635, 1995.

\bibitem{cartaneilenberg}
H.~Cartan and S.~Eilenberg.
\newblock {\em Homological algebra}.
\newblock Princeton University Press, Princeton, NJ, 1956.

\bibitem{CTZ-burk}
I.~Cheltsov, Yu. Tschinkel, and Zh. Zhang.
\newblock Equivariant geometry of the {S}egre cubic and the {B}urkhardt
  quartic.
\newblock {\em Selecta Math. (N.S.)}, 31(1):Paper No. 7, 36, 2025.

\bibitem{CTZ-uni}
I.~Cheltsov, Yu. Tschinkel, and Zh. Zhang.
\newblock Equivariant unirationality of {F}ano threefolds, 2025.
\newblock {\tt arXiv:2502.19598}.

\bibitem{DI}
I.~V. Dolgachev and V.~A. Iskovskikh.
\newblock Finite subgroups of the plane {Cremona} group.
\newblock In {\em Algebra, arithmetic, and geometry. In honor of Yu. I. Manin
  on the occasion of his 70th birthday. Vol. I}, pages 443--548. Boston, MA:
  Birkh{\"a}user, 2009.

\bibitem{Duncaness}
A.~Duncan.
\newblock Finite groups of essential dimension 2.
\newblock {\em Comment. Math. Helv.}, 88(3):555--585, 2013.

\bibitem{Duncan}
A.~Duncan.
\newblock Equivariant unirationality of del {P}ezzo surfaces of degree 3 and 4.
\newblock {\em Eur. J. Math.}, 2(4):897--916, 2016.

\bibitem{DR}
A.~Duncan and Z.~Reichstein.
\newblock Versality of algebraic group actions and rational points on twisted
  varieties.
\newblock {\em J. Algebr. Geom.}, 24(3):499--530, 2015.

\bibitem{generalov}
A.~I. Generalov.
\newblock Hochschild cohomology for the integer group ring of the semidihedral
  group.
\newblock {\em Zap. Nauchn. Sem. S.-Peterburg. Otdel. Mat. Inst. Steklov.
  (POMI)}, 388:119--151, 310--311, 2011.

\bibitem{HT-Q}
B.~Hassett and Yu. Tschinkel.
\newblock Rationality of complete intersections of two quadrics over nonclosed
  fields.
\newblock {\em Enseign. Math.}, 67(1-2):1--44, 2021.
\newblock With an appendix by Jean-Louis Colliot-Th\'el\`ene.

\bibitem{HT-torsor}
B.~Hassett and Yu. Tschinkel.
\newblock Torsors and stable equivariant birational geometry.
\newblock {\em Nagoya Math. J.}, 250:275--297, 2023.

\bibitem{hoshiyamasaki}
A.~Hoshi and A.~Yamasaki.
\newblock Rationality problem for algebraic tori.
\newblock {\em Mem. Amer. Math. Soc.}, 248(1176):v+215, 2017.

\bibitem{KT-toric}
A.~Kresch and Yu. Tschinkel.
\newblock Equivariant {B}urnside groups and toric varieties.
\newblock {\em Rend. Circ. Mat. Palermo (2)}, 72(5):3013--3039, 2023.

\bibitem{KT-toric2}
A.~Kresch and Yu. Tschinkel.
\newblock Equivariant unirationality of toric varieties, 2025.
\newblock preprint.

\bibitem{KT-Brauer}
A.~Kresch and Yu. Tschinkel.
\newblock Unramified {B}rauer group of quotient spaces by finite groups.
\newblock {\em J. Algebra}, 664:75--100, 2025.

\bibitem{kunyavskii}
B.~\`E. Kunyavski\u\i.
\newblock Three-dimensional algebraic tori.
\newblock In {\em Investigations in number theory ({R}ussian)}, pages 90--111.
  Saratov. Gos. Univ., Saratov, 1987.
\newblock Translated in Selecta Math.\ Soviet.\ {\bf 9} (1990), no.\ 1, 1--21.

\bibitem{lemire}
N.~Lemire, V.~L. Popov, and Z.~Reichstein.
\newblock Cayley groups.
\newblock {\em J. Amer. Math. Soc.}, 19(4):921--967, 2006.

\bibitem{PSY}
A.~Pinardin, A.~Sarikyan, and E.~Yasinsky.
\newblock Linearization problem for finite subgroups of the plane {C}remona
  group, 2024.
\newblock {\tt arXiv:2412.12022}.

\bibitem{tz}
Yu. Tschinkel and Zh. Zhang.
\newblock Cohomological obstructions to equivariant unirationality, 2025.
\newblock {\tt arXiv:2504.10204}.

\end{thebibliography}

\end{document}